\documentclass[10pt]{article}
\makeatletter
\renewcommand{\thefootnote}{}

\usepackage{pdfsync}
\usepackage{cite}
\usepackage{amscd}
\usepackage{amsmath}
\usepackage{latexsym}
\usepackage{amsfonts}
\usepackage{mathtools}
\usepackage{amssymb}
\usepackage{amsthm}
\usepackage{graphicx}
\usepackage{indentfirst}
\usepackage{enumerate}
\usepackage{amstext}
\usepackage[dvipdfm,
            pdfstartview=FitH,
            CJKbookmarks=true,
            bookmarksnumbered=true,
            bookmarksopen=true,
            colorlinks=true, 
            pdfborder=001,   
            citecolor=magenta,
            linkcolor=blue,
            ]{hyperref}       

\allowdisplaybreaks[4]

\newtheorem{thm}{\textbf Theorem}[section]
\newtheorem{lem}{\textbf Lemma}[section]
\newtheorem{rem}{\textbf Remark}[section]
\newtheorem{cor}{\textbf Corollary}[section]

\numberwithin{equation}{section}

\newcommand{\be}{\begin{eqnarray}}
\newcommand{\ee}{\end{eqnarray}}

\newcommand{\bes}{\begin{eqnarray*}}
\newcommand{\ees}{\end{eqnarray*}}

\def\v{\vspace{0.1in}}

\def\be{\begin{equation}}
\def\ee{\end{equation}}
\def\bee{\begin{eqnarray}}
\def\ene{\end{eqnarray}}
\def\bes{\begin{subequations}}
\def\ees{\end{subequations}}

\begin{document}

\baselineskip=13pt
\renewcommand {\thefootnote}{\dag}
\renewcommand {\thefootnote}{\ddag}
\renewcommand {\thefootnote}{ }

\pagestyle{plain}

\begin{center}
\baselineskip=16pt \leftline{} \vspace{-.3in} {\Large \bf The Cauchy problem and wave-breaking  phenomenon for a generalized sine-type FORQ/mCH equation } \\[0.2in]
\end{center}

\begin{center}
{\bf Guoquan Qin$^{\rm 1,2}$, Zhenya Yan$^{\rm 2,3,*}$, Boling Guo$^{\rm 4}$}\footnote{$^{*}$Corresponding author. {\it Email address}: zyyan@mmrc.iss.ac.cn (Z. Yan). }  \\[0.1in]
{\it\small  $^1$National Center for Mathematics and Interdisciplinary Sciences, Academy of Mathematics and Systems Science, Chinese Academy of Sciences, Beijing 100190, China \\
$^2$Key Laboratory of Mathematics Mechanization, Academy of Mathematics and Systems Science, \\ Chinese Academy of Sciences, Beijing 100190, China \\
 $^3$School of Mathematical Sciences, University of Chinese Academy of Sciences, Beijing 100049, China\\
 $^4$Institute of Applied Physics and Computational Mathematics, Beijing 100088, PR China} \\
\end{center}

\vspace{0.4in}

{\baselineskip=13pt


\vspace{-0.28in}

\noindent {\bf Abstract}\, {\small In this paper, we are concerned with   the Cauchy problem and wave-breaking  phenomenon
for a sine-type modified  Camassa-Holm (alias sine-FORQ/mCH) equation. Employing the transport equations theory and the Littlewood-Paley theory, we first establish  the local well-posedness for the strong solutions of the sine-FORQ/mCH equation in Besov spaces.
In light of the   Moser-type estimates, we are able  to derive the  blow-up criterion and  the precise blow-up quantity
 of this equation in  Sobolev spaces.  We  then give a sufficient condition with respect to the initial  data to ensure  the occurance of the  wave-breaking phenomenon by  trace the precise blow-up quantity along the characteristics associated with this equation.}

\vspace{0.15in}


 \vspace{0.1in}


\baselineskip=13pt


\section{Introduction}
\setcounter{equation}{0}
 The  Camassa-Holm (CH) equation~\cite{CamassaHolm1993PRL}
\begin{eqnarray}\label{CH}
m_{t}
+u m_{x}+2u_{x} m=0, \quad m=u-u_{xx},
\end{eqnarray}
which  describes the unidirectional propagation of shallow water waves over a flat bottom \cite{CamassaHolm1993PRL,
DullinGottwaldHolm2001PRL,Johnson2002JFM} or the propagation of axially symmetric waves in hyperelastic rods
\cite{ConstantinStrauss2000PLA, Dai1998AM},
has been extensively investigated in the literature.
(\ref{CH}) can  be  derived  from
 the celebrated  Korteweg-de Vries (KdV) equation by means of the tri-Hamiltonian duality~\cite{OlverRosenau1996PRE} and
 has some similar properties to the
 KdV equation.
For instance, it has infinitely many conservation laws
and is  completely integrable
\cite{CamassaHolm1993PRL,FokasFuchssteiner1981PD}.
Also, it can be recast as the
 following  bi-Hamilton structure~\cite{CamassaHolm1993PRL}
 \begin{eqnarray*}
m_{t}=J_1\frac{\delta H_{CH,1}}{\delta m}=K_1\frac{\delta H_{CH,2}}{\delta m},\quad J_1=-\frac12 (m \partial+\partial m),\quad K_1=\partial^3-\partial
 \end{eqnarray*}
where
\begin{eqnarray*}
H_{CH,1}=\int_{\mathbb{R}}mu \mbox{d}x,\qquad
H_{CH,2}=\frac{1}{2}\int_{\mathbb{R}}(u^{3}+uu_{x}^{2}) \mbox{d}x.
 \end{eqnarray*}
Eq.~(\ref{CH}) also admits the Lax pair~\cite{CamassaHolm1993PRL},
which allows us  to solve the Cauchy problem of (\ref{CH})
by the inverse scattering transform (IST)~\cite{ConstantinGerdjikovIvanov2006IP}
as well as to  construct the   action angle variables of (\ref{CH}) \cite{BealsSattingerSzmigielski1998AM,
BoutetKostenkoShepelskyTeschl2009SJMA,  ConstantinGerdjikovIvanov2006IP,ConstantinMcKean1999CPAM}.
However, there  are some  significant difference between (\ref{CH}) and KdV equation.
For example,
Eq.~(\ref{CH}) can model the
so-called   wave-breaking  phenomenon in nature,
namely, the wave itself  remains bounded, while its slope becomes unbounded in finite time
\cite{BealsSattingerSzmigielski2000AM,ConstantinEscher1998AM,Whitham1980},
which  can not be derived from the KdV equation.
The problem concerning  how the solution develops after the wave-breaking  has been researched  in \cite{BressanConstantin2007ARMA,
HoldenRaynaud2007CPDE}
and \cite{BressanConstantin2007AA,
HoldenRaynaud2009DCDS},
where  global conservative solutions
or global dissipative solutions are constructed,
respectively.
The CH equation (\ref{CH}) also admits a kind of solutions called peakons
\cite{BealsSattingerSzmigielski2000AM,
CamassaHolm1993PRL,
Constantin2006IM,
ConstantinEscher2007BAMC,
ConstantinEscher1996AM,
ConstantinStrauss2000CPAM,
Lenells2005JDE}
whose shapes  are stable
under small perturbations
\cite{ConstantinStrauss2000CPAM,
Lenells2004IMRN}.
(\ref{CH}) is also related to geometry.
 Firstly, it represents
the geodesic flows.
In the aperiodic case and when some asymptotic conditions are satisfied at
infinity,
equation (\ref{CH}) can be viewed as  the geodesic flow on
a manifold of diffeomorphim of the line
\cite{Constantin2000AIF}.
In the periodic case,
it can represent
the geodesic flow on the diffeomorphism group of the circle
\cite{ConstantinKolev2003CMH}.
Secondly,
it can represent the families of
 pseudo-spherical surfaces
 \cite{GorkaReyes2011IMRN,
 Reyes2002LMP,
 GuiLiuOlverQu2013CMP}.
Thirdly, it appears  from a
non-stretching invariant planar curve flow in the
 centro-equiaffine geometry
  \cite{ChouQu2002PD,
  GuiLiuOlverQu2013CMP}.
Other mathematical facts about equation
(\ref{CH}) include the existence of a
recursion operator
\cite{CamassaHolm1993PRL,
CamassaHolm1994AAM,
Fokas1995PD,
FokasFuchssteiner1981PD,
Fuchssteiner1996PD},
and so on.
The local-in-time well-posedness for the initial value problem of (\ref{CH}) has been
achieved~\cite{ConstantinEscher1998CPAM,
Danchin2001DIE}.
Eq.~(\ref{CH}) also admits global strong solutions~\cite{Constantin2000AIF,
ConstantinEscher1998CPAM}, and finite time blow-up strong solutions
\cite{Constantin2000AIF,
ConstantinEscher1998CPAM,
ConstantinEscher1998AM}.

The nonlinear terms in Eq.~(\ref{CH}) are   quadratic,
while, there do exist other CH-type equations with cubic nonlinearity, for instance,
the following modified-CH (mCH) equation (alias the Fokas-Olver-Rosenau-Qiao (FORQ) equation)~\cite{Fokas1995PD,Fuchssteiner1996PD,OlverRosenau1996PRE,Qiao2006JMP}
\begin{eqnarray}\label{mCH}
m_{t}+[(u^{2}-u_{x}^{2}) m]_{x}=0, \quad m=u-u_{xx},
\end{eqnarray}
which  was first derived
  by operating  the tri-Hamiltonian duality to the
  bi-Hamiltonian representation
   of the modified KdV equation
\cite{Fuchssteiner1996PD,
  OlverRosenau1996PRE}.
  Later, it was rederived
from the two-dimensional Euler equations
\cite{Qiao2006JMP}.
It is also completely integrable.
Its  bi-Hamiltonian structure reads
\cite{OlverRosenau1996PRE,
QiaoLi2011TMP,
GuiLiuOlverQu2013CMP}
\begin{eqnarray*}
m_{t}=J_2 \frac{\delta H_{mCH,1}}{\delta m}=K_2 \frac{\delta H_{mCH,2}}{\delta m},\quad
J_2=-\partial_{x} m \partial_{x}^{-1} m \partial_{x},\quad K_2=\partial_{x}^{3}-\partial_{x},
\end{eqnarray*}
 where the  Hamiltonians are
 \begin{eqnarray*}
H_{mCH,1}=\int_{\mathbb{R}} m u \mbox{d} x, \quad H_{mCH,2}=\frac{1}{4} \int_{\mathbb{R}}(u^{4}+2 u^{2} u_{x}^{2}-\frac{1}{3} u_{x}^{4}) \mbox{d} x.
\end{eqnarray*}
Also,  (\ref{mCH}) can be solved
by using IST since
it has   the following Lax pair \cite{QiaoLi2011TMP, GuiLiuOlverQu2013CMP}.
From the viewpoint of geometry,
equation (\ref{mCH}) arises from an intrinsic
(arc-length preserving) invariant planar curve flow in Euclidean geometry
\cite{GuiLiuOlverQu2013CMP}.
The local-wellposedness, blow-up and
wave breaking problems   of equation (\ref{mCH})
has been  investigated  in
\cite{ChenLiuQuZhang2015AM,
FuGuiLiuQu2013JDE,
GuiLiuOlverQu2013CMP,
LiuQuZhang2014AA}.
The explicit form of the multipeakons of  equation (\ref{mCH})
has been discussed in
\cite{ChangSzmigielski2018CMP}.
The peakon solution of equation (\ref{mCH})
is orbital stability
 \cite{LiuLiuQu2014AM,
 QuLiuLiu2013CMP}.
 The H\"{o}lder continuity of equation (\ref{mCH})
has been addressed  in
\cite{HimonasMantzavinos2014JNS}.

In this paper, we consider the following Cauchy problem of the sine-type
modified Camassa-Holm (alias the sine-FORQ/mCH) equation~\cite{anco2019}
\begin{eqnarray}\begin{cases}\label{sin-mCH}
m_{t}
+[\mbox{sin}(u^{2}-u_{x}^{2})m]_{x}=0, \quad m= u-u_{xx}, \\
m(0, x)=m_{0}(x),
\end{cases}
\end{eqnarray}
where $u(t, x)$ denotes the fluid velocity
and $m(t, x)=u-u_{xx}$ stands for the corresponding potential density. Note that
Eq.~(\ref{sin-mCH}) admits the conserved quantity $H_{1}=\int_{\mathbb{R}} m u\mbox{d} x$. We know that
\begin{itemize}

\item {} As $ u^2-u_x^2\to 0$,
\bee
 \sin(u^2-u_x^2)\sim u^2-u_x^2,
 \ene
in which Eq.~(\ref{sin-mCH}) just reduces to the FORQ/mCH equation (\ref{mCH}).

\item {} As $0<|u^2-u_x^2|<1$,
\bee
 \sin(u^2-u_x^2)\sim \sum_{k=1}^N\frac{(-1)^{k+1}}{(2k-1)!}(u^2-u_x^2)^{2k-1}+O((u^2-u_x^2)^{2N-1}), \quad {\rm as}\quad  0<|u^2-u_x^2|<1,
 \ene
 in which Eq.~(\ref{sin-mCH}) becomes the higher-order FORQ/mCH equation.
\end{itemize}

To the best of our knowledge,
the Cauchy problem (\ref{sin-mCH})
has not been researched  yet.
We aim to explore this problem in this paper.
 First, following the spirit of
\cite{Danchin2001DIE,
Danchin2003JDE},
we will prove   the local well-posedness for the strong
solutions of equation  (\ref{sin-mCH})
 in Besov spaces by the use of   the  Besov spaces theory
and the  transport equations theory.
Second,
the  Moser-type estimates
in  Sobolev spaces enable us
to establish a blow-up criterion
and  the precise blow-up quantity
for (\ref{sin-mCH}).
We finally
propose a sufficient condition with regard to
the initial data  to ensure the occurance of the wave-breaking phenomenon.

Before stating  our main results,
we first introduce  the solution spaces
$E_{p, r}^{s}(T)$ defined as follows
\begin{equation*}
E_{p, r}^{s}(T) \triangleq\left\{\begin{array}{l}
C\left([0, T) ; B_{p, r}^{s}\right) \cap C^{1}\left([0, T) ; B_{p, r}^{s-1}\right), \text { if } r<\infty, \vspace{0.1in}\\
C_{w}\left([0, T) ; B_{p, \infty}^{s}\right) \cap C^{0,1}\left([0, T) ; B_{p, \infty}^{s-1}\right), \text { if } r=\infty.
\end{array}\right.
\end{equation*}

We next introduce some notations to be used in this paper.

\textbf{Notations.}
Let $p(x)=\frac{1}{2}e^{-|x|}(x\in\mathbb{R})$ be the fundamental solution of $1-\partial_{x}^{2}$
on $\mathbb{R}$, and two convolution operators
 $p_{\pm}$ be~\cite{ChenGuoLiuQu2016JFA}
\begin{eqnarray}
p_{\pm} * f(x)=\frac{1}{2}e^{\mp x} \int_{-\infty}^{\pm x} e^{y} f(\pm y) d y,\label{pPlus}
\end{eqnarray}
where the star denotes the spatial convolution.
Note that $p$ and $p_{\pm}$ satisfy
\begin{equation}\label{pPlusMinus}
p=p_{+}+p_{-}, \quad p_{x}=p_{-}-p_{+},\quad
p_{xx}*f=p*f-f.
\end{equation}
Let  $\mathcal{S}$ stands for the Schwartz space
and
$\mathcal{S}^{\prime}$ represents the spaces of  temperate distributions.
Let $L^{p}(\mathbb{R})$ be the Lebesgue space  equipped with the norm $\|\cdot\|_{L^{p}}$ for $1 \leq p \leq \infty$ and
  $H^{s}(\mathbb{R})$ be the Sobolev space equipped with the
norm $\|\cdot\|_{H^{s}}$ for
 $s \in \mathbb{R}$.\\

Our first  result concerning
 the local well-posedness to (\ref{sin-mCH})
 in Besov spaces reads

\begin{thm}\label{thmLocal}
Assume that  $1 \leq p, r \leq+\infty, s>\max \{2+1 / p, 5 / 2\}.$ Let $u_{0} \in B_{p, r}^{s}.$
 Then there is a time $T>0$ such that
 the Cauchy problem (\ref{sin-mCH}) admits
  a unique solution $u \in E_{p, r}^{s}(T)$.
Furthermore, the data-to-solution map $u_{0} \mapsto u$ is continuous from a neighborhood of $u_{0}$ in $B_{p, r}^{s}$ into
$$
C\left([0, T] ; B_{p, r}^{s^{\prime}}\right) \cap C^{1}\left([0, T] ; B_{p, r}^{s^{\prime}-1}\right)
$$
for each $s^{\prime}<s$ when $r=+\infty$ and $s^{\prime}=s$ when $r<+\infty$.
\end{thm}
Taking $p=r=2$ in Theorem \ref{thmLocal},
one immediately obtains the following Corollary
concerning the local well-posedness of (\ref{sin-mCH})
in the Sobolev spaces setting,
which is more convenience for us
to establish  our blow-up results.
\begin{cor}\label{corLocal}
Suppose $s>5 / 2 $  and $u_{0} \in H^{s}.$
  Then there is a time $T>0$ such that the Cauchy problem (\ref{sin-mCH}) admits a unique strong solution
  $u \in C\left([0, T] ; H^{s}\right)
  \cap C^{1}\left([0, T] ; H^{s-1}\right)$.
Moreover, the data-to-solution map $u_{0} \mapsto u$ is continuous from a neighborhood of $u_{0}$ in $H^{s}$ into $C\left([0, T] ; H^{s}\right) \cap$ $C^{1}\left([0, T] ; H^{s-1}\right)$.
\end{cor}

We next state our result on the finite time  blow-up
of the strong solution to (\ref{sin-mCH}).
\begin{thm}\label{thmBlowup}
Let $u_{0}\in H^{s}$ be given as in Corollary \ref{corLocal}
and $u$ be the corresponding solution to (\ref{sin-mCH}).
Denote by $T^{*}$ the maximal existence time, then
\begin{eqnarray}\label{blowupTime}
T^{*}< \infty\quad
\Rightarrow\quad
\int_{0}^{T^{*}}\|m\|_{L^{\infty}}^{2}
\mbox{d}t=\infty.
\end{eqnarray}
\end{thm}

\begin{rem} Theorems \ref{thmLocal} and \ref{thmBlowup} also hold true for the sine-mCH equation with the linear dispersive term
\begin{eqnarray}\begin{cases}\label{sin-mCH2}
m_{t}+\kappa u_x+[\sin(u^{2}-u_{x}^{2})m]_{x}=0, \quad m= u-u_{xx}, \\
m(0, x)=m_{0}(x),
\end{cases}
\end{eqnarray}
where $\kappa$ is the real-valued parameter.

\end{rem}

Based on Theorem \ref{thmBlowup},
we will establishing the following result
concerning the wave-breaking phenomenon
on the solution of the Cauchy problem (\ref{sin-mCH}).

\begin{thm}\label{waveBreaking}
Suppose $m_{0} \in H^{s}(\mathbb{R})$ with $s>\frac{1}{2}.$
Let  $T^{*}>0$ be the maximal existence time of strong solution $m$ to the Cauchy problem (\ref{sin-mCH}).
Let $M(t,x)$ be defined
as in (\ref{M}).
Set $\bar{M}(t)=M(t,q(t,x_{0}))$ and
$\bar{m}(t)=m(t,q(t,x_{0})),$
where $q(t, x_{0})$ is defined in (\ref{character}).
 Let $m_{0}(x) \geq 0$ for all $x \in \mathbb{R}$,
 and $m_{0}(x_{0})>0$ for some $x_{0} \in \mathbb{R}$.
If
\begin{equation}\label{assumption}
\bar{M}(0)<0 \quad \text { and }
\quad \frac{\bar{M}(0)}{\bar{m}(0)}
\xi+\frac{1}{2}C_{1}\xi^{2}+\frac{1}{\bar{m}(0)}
<0,
\end{equation}
where $C_{1}$ is defined below (\ref{M-2})
and
\begin{eqnarray}
\xi
=-\frac{\bar{M}(0)}{C_{1}\bar{m}(0)},
\end{eqnarray}
then the solution $m$ blows up at a time
$T^{*}\in (0,\xi).$
\end{thm}

The rest of this paper is organized as follows.
 In Section 2, we will briefly recall the
   properties of Besov spaces and some Lemmas
    on the transport equation theory.
Section 3 will deal with
the proof of Theorem \ref{thmLocal}.
 In Section 4, The Moser-type estimates in Sobolev spaces
 will be used to
prove Theorem \ref{thmBlowup}.
Section 5 will
provide the proof of the wave-breaking
Theorem \ref{waveBreaking} by trace the precise blow-up
quantity along the characteristic associated
to (\ref{sin-mCH}).

\section{Preliminaries}
Since the local well-posedness for the  Cauchy problem  (\ref{sin-mCH}) will be established in  Besov-type spaces, we firstly recall some basic properties about the Littlewood-Paley theory and some useful lemmas  of the transport equation  theory.

\subsection{Basics properties of the  Littlewood-Paley theory}

The definition of Besov spaces and some of their properties will be briefly exhibited in this subsection (see \cite{BahouriCheminDanchin2011,
Chemin2004CRM} for more details).

Let $B(x_{0},r)$ be the open ball centered at $x_{0}$ with
radius $r,$ $\mathcal{C}\equiv \{\xi\in \mathbb{R}^{d} | 4/3\leq|\xi|\leq 8/3\}$,
 and $\mathcal{\tilde{C}}\equiv B(0,2/3)+\mathcal{C}.$ Then there are two radial functions $\chi\in \mathcal{D}(B(0,4/3))$
 and $\varphi\in \mathcal{D}(\mathcal{C})$
 satisfying
\begin{equation*}
\left\{\begin{array}{l}
\chi(\xi)+\sum_{q \geq 0} \varphi(2^{-q} \xi)=1,\quad
1/3 \leq \chi^2(\xi)+\sum_{q \geq 0} \varphi^2(2^{-q} \xi) \leq 1 \quad (\forall \xi \in \mathbb{R}^{d}),\v\\
|q-q^{\prime}| \geq 2 \Rightarrow \operatorname{Supp} \varphi(2^{-q} \cdot) \cap \operatorname{Supp} \varphi(2^{-q^{\prime}} \cdot)=\varnothing, \v \\
q \geq 1 \Rightarrow \operatorname{Supp} \chi(\cdot) \cap \operatorname{Supp} \varphi(2^{-q^{\prime}} \cdot)=\varnothing,\quad
|q-q^{\prime}| \geq 5 \Rightarrow 2^{q^{\prime}} \widetilde{\mathcal{C}} \cap 2^{q} \mathcal{C}=\varnothing.
\end{array}\right.
\end{equation*}

The dyadic operators $\Delta_{q}$ and $S_{q}$ acting on $u(t,x)\in S'(\mathbb{R}^d)$ are defined as
\begin{equation*}
\begin{array}{l}
\Delta_{q} u=\left\{
 \begin{array}{ll}
   0, & q \leq-2, \v \\
 \chi(D) u=\int_{\mathbb{R}^{d}} \tilde{h}(y) u(x-y)dy, & q=-1, \v \\
 \varphi\left(2^{-q} D\right) u=2^{q d} \int_{\mathbb{R}^{d}} h\left(2^{q} y\right) u(x-y) dy, & q \geq 0,
 \end{array}\right. \v \\
S_{q} u=\sum_{q^{\prime} \leq q-1} \Delta_{q^{\prime}} u,
\end{array}
\end{equation*}
where
$h = \mathcal{F}^{-1} \varphi$ and $ \tilde{h} = \mathcal{F}^{-1} \chi$ with $\mathcal{F}^{-1}$ denoting the inverse Fourier transform.

The Besov spaces is
$B_{p, r}^{s}(\mathbb{R}^{d})=\left\{u \in S^{\prime}\, \big|\, \|u\|_{B_{p, r}^{s}(\mathbb{R}^{d})}=\big(\sum_{j \geq-1} 2^{r j s}\|\Delta_{j} u\|_{L^{p}(\mathbb{R}^{d})}^{r}\big)^{1/r}
<\infty\right\}.$
With the above-defined Besov spaces, we next recall some of their properties.

\begin{lem} {\rm (Embedding property)}\cite{BahouriCheminDanchin2011,Chemin2004CRM}\label{21lem1}
Suppose $1 \leq p_{1} \leq p_{2} \leq \infty$, $1 \leq r_{1} \leq r_{2} \leq \infty$ and $s$ be real.
Then it holds that $B_{p_{1}, r_{1}}^{s}(\mathbb{R}^{d}) \hookrightarrow B_{p_{2}, r_{2}}^{s-d(1/p_1-1/p_2)}
(\mathbb{R}^{d})$. If $s>d/p$ or $s=d/p,\, r=1,$ then there holds $B_{p, r}^{s}(\mathbb{R}^{d}) \hookrightarrow L^{\infty}(\mathbb{R}^{d})$.
\end{lem}

\begin{lem}{\rm (Interpolation)}\cite{BahouriCheminDanchin2011,Chemin2004CRM} \label{21lem2}
Let $s_1,\, s_2$  be real numbers with $s_{1}<s_{2}$ and $\theta \in(0,1).$
 Then  there exists a constant $C$ such that
\begin{eqnarray*}
\|u\|_{B_{p, r}^{\theta s_{1}+(1-\theta) s_{2}}}
\leq\|u\|_{B_{p, r}^{s_{1}}}^{\theta}\|u\|_{B_{p, r}^{s_{2}}}^{(1-\theta)},\quad
\|u\|_{B_{p, 1}^{\theta s_{1}+(1-\theta) s_{2}}}
\leq \frac{C}{s_{2}-s_{1}}\frac{1}{\theta (1-\theta)}
\|u\|_{B_{p, \infty}^{s_{1}}}^{\theta}
\|u\|_{B_{p, \infty}^{s_{2}}}^{(1-\theta)},
\end{eqnarray*}
where $(p, r) \in[1, \infty]^{2}$.
\end{lem}

\begin{lem}{\rm (Product law)}\cite{BahouriCheminDanchin2011,Chemin2004CRM}\label{21lem3}
Let  $(p, r)\in [1, \infty]^{2}$
and $s$ be real. Then
$
\|u v\|_{B_{p, r}^{s}(\mathbb{R}^{d})} \leq C(\|u\|_{L^{\infty}(\mathbb{R}^{d})}\|v\|_{B_{p, r}^{s}(\mathbb{R}^{d})}
+\|u\|_{B_{p, r}^{s}(\mathbb{R}^{d})}
\|v\|_{L^{\infty}(\mathbb{R}^{d})}),
$
namely,  the space $L^{\infty}(\mathbb{R}^{d}) \cap B_{p, r}^{s}(\mathbb{R}^{d})$ is an algebra.
Moreover, if  $s>d/p$ or $s=d/p,\, r=1,$ then there holds
$
\|u v\|_{B_{p, r}^{s}\left(\mathbb{R}^{d}\right)} \leq C\|u\|_{B_{p, r}^{s}\left(\mathbb{R}^{d}\right)}\|v\|_{B_{p, r}^{s}\left(\mathbb{R}^{d}\right)}.
$
\end{lem}

\begin{lem}{\rm (Moser-type estimates)}\cite{BahouriCheminDanchin2011,Danchin2001DIE}         \label{21lem4}
Let $s>\max \{d/p,\, d/2\}$
and $(p, r)\in[1, \infty]^{2}$.
Then, for any $a \in B_{p, r}^{s-1}(\mathbb{R}^{d})$ and $b \in B_{p, r}^{s}(\mathbb{R}^{d}),$ there holds
$
\|a b\|_{B_{p, r}^{s-1}(\mathbb{R}^{d})} \leq C\|a\|_{B_{p, r}^{s-1}(\mathbb{R}^{d})}\|b\|_{B_{p, r}^{s}(\mathbb{R}^{d})}.
$
\end{lem}

The following  Lemma is useful  for proving the blow-up criterion.
\begin{lem}{\rm (Moser-type estimates)}\cite{GuiLiu2010JFA,
GuiLiuOlverQu2013CMP}  \label{21lem5} Let $s \geq 0.$ Then one has
$$
\begin{aligned}
\|f g\|_{H^{s}(\mathbb{R})}
& \leq C(\|f\|_{H^{s}(\mathbb{R})}\|g\|_{L^{\infty}(\mathbb{R})}
+\|f\|_{L^{\infty}(\mathbb{R})}\|g\|_{H^{s}(\mathbb{R})}),\v \\
\|f \partial_{x} g\|_{H^{s}(\mathbb{R})}
 & \leq C(\|f\|_{H^{s+1}(\mathbb{R})}\|g\|_{L^{\infty}(\mathbb{R})}
 +\|f\|_{L^{\infty}(\mathbb{R})}\|\partial_{x}g\|_{H^{s}(\mathbb{R})}),
\end{aligned}
$$
where $C$'s are constants independent of $f$ and $g$.
\end{lem}

\begin{lem}\cite{BahouriCheminDanchin2011}\label{besovProperty}
 Let $s \in \mathbb{R}$ and $1 \leq p, r \leq \infty$.
Then the Besov spaces have the following properties:
\begin{itemize}
\item{}$B_{p, r}^{s}(\mathbb{R}^{d})$ is a Banach space and continuously embedding into $\mathcal{S}^{\prime}(\mathbb{R}^{d}),$ where $\mathcal{S}^{\prime}(\mathbb{R}^{d})$ is the dual space of the Schwartz space $\mathcal{S}(\mathbb{R}^{d})$;

\item{}If $p, r<\infty,$ then $\mathcal{S}(\mathbb{R}^{d})$ is dense in $B_{p, r}^{s}(\mathbb{R}^{d})$;

\item {} If $u_{n}$ is a bounded sequence of $B_{p, r}^{s}(\mathbb{R}^{d}),$ then an element $u \in B_{p, r}^{s}(\mathbb{R}^{d})$ and a subsequence $u_{n_{k}}$ exist such that
$\lim _{k \rightarrow \infty} u_{n_{k}}=u \text { in } \mathcal{S}^{\prime}(\mathbb{R}^{d}) \text { and }\|u\|_{B_{p, r}^{s}(\mathbb{R}^{d})} \leq C \liminf _{k \rightarrow \infty}\|u_{n_{k}}\|_{B_{p, r}^{s}(\mathbb{R}^{d})}.$
\end{itemize}
\end{lem}

\subsection{Some lemmas in the theory of the transport equation}

We recall some a priori estimates\cite{BahouriCheminDanchin2011,
Danchin2001DIE} for the following transport equation
\begin{equation}\label{transport}
\left\{\begin{array}{l}
f_{t}+v\cdot \nabla f=g, \v \\
\left.f\right|_{t=0}=f_{0}.
\end{array}\right.
\end{equation}

\begin{lem}\cite{BahouriCheminDanchin2011,
Danchin2001DIE}\label{22lem1} Let $1 \leq p \leq p_{1} \leq \infty,$ $1 \leq r \leq \infty$
and $s \geq -d \min (1/p_1,\, 1-1/p)$.
Let  $f_{0} \in B_{p, r}^{s}(\mathbb{R}^{d})$.
$g \in L^{1}([0, T] ; B_{p, r}^{s}(\mathbb{R}^{d}))$ and $\nabla v \in L^{1}([0, T] ; B_{p, r}^{s}(\mathbb{R}^{d}) \cap L^{\infty}(\mathbb{R}^{d})),$ then there exists a unique solution
 $f \in L^{\infty}([0, T] ; B_{p, r}^{s}(\mathbb{R}^{d}))$  to Eq.~(\ref{transport})  satisfying:
\begin{eqnarray}
\|f\|_{B_{p, r}^{s}(\mathbb{R}^{d})}
 \leq\|f_{0}\|_{B_{p, r}^{s}
(\mathbb{R}^{d})}
+\int_{0}^{t}[\|g(t^{\prime})\|_{B_{p, r}^{s}
(\mathbb{R}^{d})}
+C V_{p_{1}}(t^{\prime})\|f(t^{\prime})\|_{B_{p, r}^{s}
(\mathbb{R}^{d})}] d t^{\prime}, \label{priori1}\\
\|f\|_{B_{p, r}^{s}(\mathbb{R}^{d})}
\leq\left[\|f_{0}\|_{B_{p, r}^{s}
(\mathbb{R}^{d})}+\int_{0}^{t}\|g(t^{\prime})\|_{B_{p, r}^{s}
(\mathbb{R}^{d})}e^{-C V_{p_{1}}(t^{\prime})} d t^{\prime}\right]e^{C V_{p_1}(t)},\label{priori2}
\end{eqnarray}
where $V_{p_{1}}(t)=\int_{0}^{t}\|\nabla v\|_{B_{p_{1}, \infty}^{d/p_{1}}(\mathbb{R}^{d}) \cap L^{\infty}(\mathbb{R}^{d})} d t^{\prime}$ if $s<1+d/p_{1}, V_{p_{1}}(t)=\int_{0}^{t}\|\nabla v\|_{B_{p_{1}, r}^{s-1}\left(\mathbb{R}^{d}\right)} d t^{\prime}$
if $s>1+d/p_{1}$ or $s=1+d/p_{1},\, r=1,$ and $C$ is a constant depending only on $s, p, p_{1}$, and $r$.
\end{lem}

\begin{lem}\cite{BahouriCheminDanchin2011}\label{22lem4} Let $s \geq -d \min (1/p_1, 1-1/p).$
Let $f_{0} \in B_{p, r}^{s}(\mathbb{R}^{d})$, $g \in L^{1}([0, T] ; B_{p, r}^{s}(\mathbb{R}^{d}))$  and $v \in L^{\rho}([0, T]; B_{\infty, \infty}^{-M}(\mathbb{R}^{d}))$  for some $\rho>1$ and $M>0$  be a time-dependent vector field satisfying
$$
\nabla v\in\left\{\begin{array}{ll}
L^{1}([0, T] ; B_{p_{1}, \infty}^{d/p}(\mathbb{R}^{d})), &  {\rm if }\,\, s<1+d/p_{1}, \v \\
L^{1}\left([0, T] ; B_{p_{1}, \infty}^{s-1}\left(\mathbb{R}^{d}\right)\right), & {\rm if }\,\, s>1+d/p_1\,\,  {\rm or } \,\, s=1+d/p_1 \,\, {\rm and } \,\, r=1.
\end{array}\right.
$$
Then, Eq.~(\ref{transport}) has a unique solution $f\in \mathcal{C}([0, T] ; B_{p, r}^{s}(\mathbb{R}^{d}))$ for $r<\infty$, or
$f\in (\bigcap_{s^{\prime}<s} \mathcal{C}([0, T]; B_{p, \infty}^{s^{\prime}}(\mathbb{R}^{d}))) \cap \mathcal{C}_{w}([0, T] ; B_{p, \infty}^{s}(\mathbb{R}^{d})))$ for $r=\infty$. Furthermore, the inequalities
(\ref{priori1})-(\ref{priori2}) hold.
\end{lem}

\begin{lem} {\rm (A priori estimate in the Sobolev spaces)}\cite{BahouriCheminDanchin2011,
GuiLiu2010JFA}\label{22lem5} Let $0\leq\sigma<1$.
Let $f_{0} \in H^{\sigma},\, g \in L^{1}(0, T ; H^{\sigma})$ and $\partial_{x} v \in L^{1}(0, T ; L^{\infty}).$ Then the solution $f$
to Eq.~(\ref{transport}) belongs to $C([0, T] ; H^{\sigma}).$
  More precisely,
 there is a constant $C$
depending only on $\sigma$ such that
$$
\|f\|_{H^{\sigma}}\leq\|f_{0}\|_{H^{\sigma}}
+\int_{0}^{t}[\|g(\tau)\|_{H^{\sigma}}+CV^{\prime}(\tau)\|f(\tau)\|_{H^{\sigma}}] d\tau,\quad V(t)=\int_{0}^{t}\|\partial_{x} v(\tau)\|_{L^{\infty}} \mbox{d} \tau.
$$
\end{lem}

\section{Local well-posedness}

In this section,
we will give the proof of
the local well-posedness in Besov spaces to (\ref{sin-mCH}),
namely, Theorem \ref{thmLocal}.

\begin{proof}
We first employ the classical
Friedrichs regularization method
to construct the approximate solutions
to (\ref{sin-mCH}). Let $m^{(l+1)}$ solve   the following
linear transport equation  inductively
\begin{eqnarray}\label{lplus1-th}
\begin{cases}
 \partial_{t}m^{(l+1)}
+\mbox{sin}[(u^{(l)})^{2}-(u_{x}^{(l)})^{2}]\partial_{x}m^{(l+1)}
=-2\mbox{cos}[(u^{(l)})^{2}-(u_{x}^{(l)})^{2}]
u_{x}^{(l)}(m^{(l)})^{2},\\
m^{(l+1)}_{t=0}=m_{0}^{(l+1)}(x)=S_{l+1} m_{0},
\end{cases}
\end{eqnarray}
where $m^{(0)} \coloneqq 0$.

Suppose $m^{(l)} \in L^{\infty}(0, T ; B_{p, r}^{s-2}),$ where
$s-2>\max \{\frac{1}{p}, \frac{1}{2}\}$
and consequently
$B_{p, r}^{s-2}$ is an algebra.
So the right hand side of
Eq.~(\ref{lplus1-th}) is in
$L^{\infty}\left(0, T ; B_{p, r}^{s-2}\right).$
Hence, by Lemma \ref{22lem4} and the high regularity of $u,$
Eq.~(\ref{lplus1-th})
admits a global solution $m^{(l+1)}\in E_{p, r}^{s-2}$ for all positive $T$.

Applying (\ref{priori2}) in Lemma \ref{22lem1} to Eq.~(\ref{lplus1-th}) yields
\begin{eqnarray}\label{lplus1-th-1}
&&\|m^{(l+1)}(t)\|_{B^{s-2}_{p,r}}
\leq e^{C\int_{0}^{t}
\|\mbox{sin}[(u^{(l)})^{2}-(u_{x}^{(l)})^{2}]\|_{B^{s-2}_{p,r}}\mbox{d}\tau}
\|m_{0}\|_{B^{s-2}_{p,r}}\quad\quad\quad\quad\nonumber\\
&&\quad\quad\quad\quad\quad\quad
+C\int_{0}^{t}e^{C\int_{\tau}^{t}
\|\mbox{sin}[(u^{(l)})^{2}-(u_{x}^{(l)})^{2}]\|_{B^{s-2}_{p,r}}\mbox{d}\tau^{\prime}}
\|2\mbox{cos}[(u^{(l)})^{2}-(u_{x}^{(l)})^{2}]
u_{x}^{(l)}(m^{(l)})^{2}\|_{B^{s-2}_{p,r}}\mbox{d}\tau.
\end{eqnarray}
for  $l=0,1,2,\cdots.$

Using
Lemma \ref{21lem3} concerning
the product law in Besov spaces, one obtains
\begin{eqnarray}
\|\mbox{sin}[(u^{(l)})^{2}-(u_{x}^{(l)})^{2}]\|_{B^{s-2}_{p,r}}
\leq \|(u^{(l)})^{2}-(u_{x}^{(l)})^{2}\|_{B^{s-2}_{p,r}}
\leq \|u^{(l)}\|_{B^{s}_{p,r}}^{2},\label{lplus1-th-2}\\
\|2\mbox{cos}[(u^{(l)})^{2}-(u_{x}^{(l)})^{2}]
u_{x}^{(l)}(m^{(l)})^{2}\|_{B^{s-2}_{p,r}}
\leq C\|u_{x}^{(l)}(m^{(l)})^{2}\|_{B^{s-2}_{p,r}}
\leq C\|u^{(l)}\|_{B^{s}_{p,r}}^{3}.\label{lplus1-th-3}
\end{eqnarray}
Substituting (\ref{lplus1-th-2})-(\ref{lplus1-th-3}) to
(\ref{lplus1-th}), one finds
\begin{eqnarray}\label{lplus1-th-4}
\|u^{(l+1)}(t)\|_{B^{s}_{p,r}}
\leq e^{C\int_{0}^{t}\|u^{(l)}\|_{B^{s}_{p,r}}^{2}\mbox{d}\tau}
\|u_{0}\|_{B^{s}_{p,r}}
+C\int_{0}^{t}e^{C\int_{\tau}^{t}
\|u^{(l)}\|_{B^{s}_{p,r}}^{2}
\mbox{d}\tau^{\prime}}
\|u^{(l)}\|_{B^{s}_{p,r}}^{3}
\mbox{d}\tau.
\end{eqnarray}
Now, assume $\|u^{(l)}(t)\|_{B^{s}_{p,r}}\leq a(t)$.
Plugging this assumption into (\ref{lplus1-th-4})  leads to
\begin{eqnarray}\label{lplus1-th-5}
\|u^{(l+1)}(t)\|_{B^{s}_{p,r}}
\leq e^{C\int_{0}^{t}a^{2}\mbox{d}\tau}
\|u_{0}\|_{B^{s}_{p,r}}
+C\int_{0}^{t}e^{C\int_{\tau}^{t}
a^{2}\mbox{d}\tau^{\prime}}a^{3}\mbox{d}\tau.
\end{eqnarray}
Let the right hand side of inequality (\ref{lplus1-th-5})
be equal to $a$, we obtain
\begin{eqnarray*}
\begin{cases}
 \dot{a}=2Ca^{3},\\
 a(0)=\|u_{0}\|_{B^{s}_{p,r}}.
\end{cases}
\end{eqnarray*}
Solving this ordinary differential equation,
one deduces
\begin{eqnarray*}
a(t)=\|u_{0}\|_{B^{s}_{p,r}}
[1-4Ct\|u_{0}\|_{B^{s}_{p,r}}^{2}]^{-1/2}.
\end{eqnarray*}
Therefore, we conclude that the solution sequence
 $\{u^{(l)}\}_{l=1}^{\infty}$
of equation (\ref{lplus1-th}) is uniformly bounded
in $C([0,T]; B^{s}_{p,r})$ with
\begin{eqnarray*}
T< \frac{1}{4C\|u_{0}\|_{B^{s}_{p,r}}^{2}}.
\end{eqnarray*}

Next, we shall prove that $\{m^{(l+1)}\}_{l=1}^{\infty}$
is a Cauchy sequence  in  $C([0,T]; B^{s-3}_{p,r})$.

In fact, from equation (\ref{lplus1-th}), one derives
\begin{eqnarray}\label{difference}
&&\partial_{t}[m^{(l+i+1)}-m^{(l+1)}]
+\mbox{sin}[(u^{l+i})^{2}-(u_{x}^{l+i})^{2}]
\partial_{x}[m^{(l+i+1)}-m^{(l+1)}]\nonumber\\
&&=\bigg\{\mbox{sin}[(u^{(l)})^{2}-(u_{x}^{(l)})^{2}]
-\mbox{sin}[(u^{(l+i)})^{2}-(u_{x}^{(l+i)})^{2}]
\bigg\}\partial_{x}m^{(l+1)}\nonumber\\
&&-2\bigg\{
\mbox{cos}[(u^{(l+i)})^{2}-(u_{x}^{(l+i)})^{2}]u_{x}^{(l+i)}(m^{(l+i)})^{2}
-\mbox{cos}[(u^{(l)})^{2}-(u_{x}^{(l)})^{2}]u_{x}^{(l)}(m^{(l)})^{2}
\bigg\}\coloneqq g.
\end{eqnarray}
As a consequence of  Lemma \ref{22lem1},
one deduces
\begin{eqnarray}\label{difference-1}
&&\|m^{(l+i+1)}-m^{(l+1)}\|_{B^{s-3}_{p,r}}
\leq \exp\bigg\{C\int_{0}^{t}
\|\mbox{sin}[(u^{l+i})^{2}-(u_{x}^{l+i})^{2}]\|_{B^{s-3}_{p,r}}
\mbox{d}\tau\bigg\}\\
&&\times\bigg\{\|m^{(l+i+1)}-m^{(l+1)}\|_{B^{s-3}_{p,r}}
+\int_{0}^{t}\exp\bigg\{-C\int_{0}^{\tau}
\|\mbox{sin}[(u^{l+i})^{2}-(u_{x}^{l+i})^{2}]\|_{B^{s-3}_{p,r}}\mbox{d}\tau^{\prime}\bigg\}
\|g\|_{B^{s-3}_{p,r}}\mbox{d}\tau\bigg\}.\nonumber
\end{eqnarray}
 Lemma \ref{21lem3} enables us to conclude
\begin{eqnarray}\label{difference-2}
&&\bigg\|\bigg\{\mbox{sin}[(u^{(l)})^{2}-(u_{x}^{(l)})^{2}]
-\mbox{sin}[(u^{(l+i)})^{2}-(u_{x}^{(l+i)})^{2}]
\bigg\}\partial_{x}m^{(l+1)}\bigg\|_{B^{s-3}_{p,r}}\nonumber\\
&&\leq C\|\partial_{x}m^{(l+1)}\|_{B^{s-3}_{p,r}}
\|\mbox{sin}[(u^{(l)})^{2}-(u_{x}^{(l)})^{2}]
-\mbox{sin}[(u^{(l+i)})^{2}-(u_{x}^{(l+i)})^{2}]\|_{B^{s-2}_{p,r}}\nonumber\\
&&\leq C\|u^{(l+1)}\|_{B^{s}_{p,r}}
\bigg\|2\mbox{cos}\left[\frac{(u^{(l+i)})^{2}-(u_{x}^{(l+i)})^{2}
+(u^{(l)})^{2}-(u_{x}^{(l)})^{2}}{2}\right]\nonumber\\
&&\quad\quad\quad\quad\quad\quad
\times\mbox{sin}\left[\frac{(u^{(l+i)})^{2}-(u_{x}^{(l+i)})^{2}
-(u^{(l)})^{2}+(u_{x}^{(l)})^{2}}{2}\right]\bigg\|_{B^{s-2}_{p,r}}\nonumber\\
&&\leq C\|u^{(l+1)}\|_{B^{s}_{p,r}}
[\|(u^{(l+i)})^{2}-(u^{(l)})^{2}\|_{B^{s-2}_{p,r}}
+\|(u_{x}^{(l+i)})^{2}-(u_{x}^{(l)})^{2}\|_{B^{s-2}_{p,r}}]\nonumber\\
&&\leq C\|u^{(l+1)}\|_{B^{s}_{p,r}}
\|u^{(l)}-u^{(l+i)}\|_{B^{s-1}_{p,r}}
(\|u^{(l)}\|_{B^{s}_{p,r}}+
\|u^{(l+i)}\|_{B^{s}_{p,r}})
\end{eqnarray}
and
\begin{eqnarray}\label{difference-3}
&&\|\mbox{cos}[(u^{(l+i)})^{2}-(u_{x}^{(l+i)})^{2}]
u_{x}^{(l+i)}(m^{(l+i)})^{2}
-\mbox{cos}[(u^{(l)})^{2}-(u_{x}^{(l)})^{2}]u_{x}^{(l)}(m^{(l)})^{2}
\|_{B^{s-3}_{p,r}}\nonumber\\
&&\leq \|\{\mbox{cos}[(u^{(l+i)})^{2}-(u_{x}^{(l+i)})^{2}]
-\mbox{cos}[(u^{(l)})^{2}-(u_{x}^{(l)})^{2}]\}
u_{x}^{(l+i)}(m^{(l+i)})^{2}\|_{B^{s-3}_{p,r}}\nonumber\\
&&\quad\quad+\|\mbox{cos}[(u^{(l)})^{2}-(u_{x}^{(l)})^{2}]
(u_{x}^{(l+i)}-u_{x}^{(l)})(m^{(l+i)})^{2}\|_{B^{s-3}_{p,r}}\nonumber\\
&&\quad\quad+\|\mbox{cos}[(u^{(l)})^{2}-(u_{x}^{(l)})^{2}]
u_{x}^{(l)}((m^{(l+i)})^{2}-(m^{(l+i)})^{2})\|_{B^{s-3}_{p,r}}\nonumber\\
&&\leq \|m^{(l+i)}\|_{B^{s-3}_{p,r}}
\|m^{(l+i)}\|_{B^{s-2}_{p,r}}
\|u_{x}^{(l+i)}\|_{B^{s-2}_{p,r}}\nonumber\\
&&\quad\quad
\times
\bigg\|2\mbox{sin}[\frac{(u^{(l+i)})^{2}-(u_{x}^{(l+i)})^{2}
+(u^{(l)})^{2}-(u_{x}^{(l)})^{2}}
{2}]\mbox{sin}[\frac{(u^{(l+i)})^{2}-(u_{x}^{(l+i)})^{2}
-(u^{(l)})^{2}+(u_{x}^{(l)})^{2}}
{2}]\bigg\|_{B^{s-2}_{p,r}}\nonumber\\
&&\quad\quad+\|m^{(l+i)}\|_{B^{s-3}_{p,r}}
\|m^{(l+i)}\|_{B^{s-2}_{p,r}}
\|u_{x}^{(l+i)}-u_{x}^{(l)}\|_{B^{s-2}_{p,r}}\nonumber\\
&&\quad\quad+\|m^{(l+i)}-m^{(l)}\|_{B^{s-3}_{p,r}}
\|m^{(l+i)}+m^{(l)}\|_{B^{s-2}_{p,r}}
\|u_{x}^{(l)}\|_{B^{s-2}_{p,r}}\nonumber\\
&&\leq C\|u^{(l+i)}\|_{B^{s}_{p,r}}^{3}
(\|(u^{(l+i)})^{2}-(u^{(l)})^{2}\|_{B^{s-2}_{p,r}}
+\|(u_{x}^{(l+i)})^{2}-(u_{x}^{(l)})^{2}\|_{B^{s-2}_{p,r}})\nonumber\\
&&\quad\quad+C\|u^{(l+i)}\|_{B^{s}_{p,r}}^{2}
\|u^{(l+i)}-u^{(l)}\|_{B^{s-1}_{p,r}}\nonumber\\
&&\quad\quad+C(\|u^{(l+i)}\|_{B^{s}_{p,r}}
+\|u^{(l)}\|_{B^{s}_{p,r}})
\|u^{(l)}\|_{B^{s-1}_{p,r}}
\|u^{(l+i)}-u^{(l)}\|_{B^{s-1}_{p,r}}\nonumber\\
&&\leq C\|u^{(l+i)}\|_{B^{s}_{p,r}}^{3}
(\|u^{(l+i)}\|_{B^{s}_{p,r}}^{2}
+\|u^{(l)}\|_{B^{s}_{p,r}}^{2})
\|u^{(l+i)}-u^{(l)}\|_{B^{s-1}_{p,r}}\nonumber\\
&&\quad\quad+C\|u^{(l+i)}\|_{B^{s}_{p,r}}^{2}
\|u^{(l+i)}-u^{(l)}\|_{B^{s-1}_{p,r}}\nonumber\\
&&\quad\quad+C\|u^{(l+i)}\|_{B^{s}_{p,r}}^{2}
\|u^{(l+i)}-u^{(l)}\|_{B^{s-1}_{p,r}}
(\|u^{(l+i)}\|_{B^{s}_{p,r}}
+\|u^{(l)}\|_{B^{s}_{p,r}}).
\end{eqnarray}
Plugging (\ref{difference-2})-(\ref{difference-3})
into (\ref{difference-1}) leads to
\begin{eqnarray}\label{difference-4}
\|u^{(l+i+1)}-u^{(l+1)}\|_{B^{s-1}_{p,r}}
\leq C_{T}\bigg(2^{-l}
+\int_{0}^{t}\|u^{(l+i)}-u^{(l)}\|_{B^{s-1}_{p,r}}\mbox{d}\tau\bigg).
\end{eqnarray}
Next, notice that
\begin{eqnarray}
\begin{array}{rl}
\|m_{0}^{(l+i+1)}-m_{0}^{(l+1)}\|_{B^{s-3}_{p,r}}
&=\|S_{l+i+1}m_{0}-S_{l+1}m_{0}\|_{B^{s-3}_{p,r}} \v\\
&=\|\sum_{q=l+1}^{l+l}\Delta_{q}m_{0}\|_{B^{s-3}_{p,r}}\leq C2^{-l}\|m_{0}\|_{B^{s-3}_{p,r}}
\end{array}
\end{eqnarray}
and $\{m^{(l)}\}$ is bounded in $C([0,T]; B_{p,r}^{s-2})$,
we thus deduce
\begin{eqnarray}
\|m^{(l+i+1)}-m^{(l+1)}\|_{B^{s-3}_{p,r}}
\leq C_{T}\left(2^{-l}
+\int_{0}^{t}\|m^{(l+l)}-m^{(l)}\|_{B^{s-3}_{p,r}}\mbox{d}\tau\right).
\end{eqnarray}
So we  conclude that
\begin{eqnarray}
\|m^{(l+i+1)}-m^{(l+1)}\|_{C(0,T; B^{s-3}_{p,r})}
\leq \frac{C_{T}}{2^l}\sum_{k=0}^{l}\frac{(2TC_{T})^{k}}{k!}
+\frac{(TC_{T})^{l+1}}{(l+1)!}\|m^{(l)}-m^{(0)}\|_{C(0,T; B^{s-3}_{p,r})}.
\end{eqnarray}
Since $\{m^{(l)}\}$ is uniformly bounded
in ${C(0,T; B^{s-3}_{p,r})},$
one can find a new constant $C_{T}^{\prime}$
so that
\begin{equation*}
\|m^{(l+i+1)}-m^{(l+1)}\|_{C(0, T ; B_{p, r}^{s-3})} \leq \frac{C_{T}^{\prime}}{2^{n}}.
\end{equation*}
Accordingly, $\{m^{(n)}\}$ is  Cauchy  in $C(0, T ; B_{p, r}^{s-3})$ and   converges to some limit function $m \in C(0, T ; B_{p, r}^{s-3}).$

We next show the existence of the solution to equation (\ref{sin-mCH}) by proving that the obtained
 limit function $m$ satisfies equation
(\ref{sin-mCH}) in the sense of distribution and belongs to
$E^{s}_{p,r}.$

Firstly, using  Lemma \ref{besovProperty}(iii)
and the uniform boundedness of $\{m^{(l)}\}$ in
$L^{\infty}(0,T; B^{s-2}_{p,r})$, one derives
that $m\in L^{\infty}(0,T; B^{s-2}_{p,r}).$

Secondly, we claim that $\{m^{(l)}\}$
converges to $m$ in $ C(0,T; B^{s^{\prime}}_{p,r})$
for all $s^{\prime} < s-2.$
In fact, this is the consequence of  the following statement: $\|m_{l}-m\|_{B_{p, r}^{s^{\prime}}} \leq C\|m_{l}-m\|_{B_{p, r}^{s-3}}$
when $s^{\prime} \leq s-3$
and $
\|m_{l}-m\|_{B_{p, r}^{\prime}}
\leq C\|m_{l}-m\|_{B_{p, r}^{s-3}}^{\theta}(\|m_{l}\|_{B_{p, r}^{s-2}}+\|m\|_{B_{p, r}^{s-2}})^{1-\theta}
$
with $\theta=s-2-s^{\prime}$
when $s-3<s^{\prime} \leq s-2$
recalling
Lemma \ref{21lem2}.
This claim enables us to take limit in equation (\ref{lplus1-th})
to conclude that the limit function $m$ does satisfy equation (\ref{sin-mCH}).

Note that equation (\ref{sin-mCH}) can be recast as a transport equation
\begin{eqnarray}\label{meq}
\partial_{t}m
+\mbox{sin}(u^{2}-u_{x}^{2})
\partial_{x}m
=-2\mbox{cos}(u^{2}-u_{x}^{2})u_{x}m.
\end{eqnarray}
Since $m\in L^{\infty}(0,T; B^{s-2}_{p,r}),$
it is easy to deduce that the right hand side of the above
equation  also belongs to $L^{\infty}(0,T; B^{s-2}_{p,r})$
in view of the product law in Besov spaces and the
Sobolev embedding.
Consequently, Lemma \ref{22lem4} implies
$m\in C([0, T) ; B_{p, r}^{s-2})$ when $r<\infty$
or
$m\in C_{w}([0, T) ; B_{p, r}^{s-2})$ when $r=\infty.$
On the other hand, from the Moser-type estimates
in Lemma \ref{21lem4},
we infer that $[\mbox{sin}(u^{2}-u_{x}^{2})
]\partial_{x}m$
is bounded in $L^{\infty}(0,T; B^{s-3}_{p,r}).$
Therefore, one knows
$\partial_{t} m\in C([0, T) ; B_{p, r}^{s-3})$
when $r<\infty$ according to the  high regularity of $u$
and equation (\ref{sin-mCH}) and thus
$m\in E^{s-2}_{p,r}.$

Furthermore, a standard
use of a sequence of viscosity approximate solutions
$\{u_{\epsilon}\}_{\epsilon>0}$ for
(\ref{sin-mCH})
which converges uniformly in
$C([0, T]; B^{s-2}_{p,r }) \cap C^{1}([0, T ]; B^{s-3}_{p, r})$
implies the continuity of the solution $m$ in
$E^{s-2}_{p,r}(T)$.

To complete the proof of Theorem \ref{thmLocal},
we next show the uniqueness.

Let $m=u-u_{xx}$ and $n=v-v_{xx}$ both be solution
to (\ref{sin-mCH}).
Then we have
\begin{eqnarray}\label{unique}
&&\partial_{t}(m-n)
+\mbox{sin}(u^{2}-u_{x}^{2})\partial_{x}(m-n)\nonumber\\
&&\qquad =-2[\mbox{cos}[u^{2}-u_{x}^{2}]u_{x}m^{2}
-\mbox{cos}(v^{2}-v_{x}^{2})v_{x}n^{2}]
-[\mbox{sin}(v^{2}-v_{x}^{2})
-\mbox{sin}(u^{2}-u_{x}^{2})n_{x}]\coloneqq f,
\end{eqnarray}
where $(m-n)(0)=m_{0}-n_{0}=0.$

Using (\ref{priori1}) in Lemma \ref{22lem1},
one finds
\begin{eqnarray}\label{unique-1}
\|m-n\|_{B^{s-3}_{p,r}}
\leq \|m_{0}-n_{0}\|_{B^{s-3}_{p,r}}
+C\int_{0}^{t}\|m-n\|_{B^{s-3}_{p,r}}
\|\partial_{x}\mbox{sin}[u^{2}-u_{x}^{2}]\|_{B^{s-2}_{p,r}}
\mbox{d}\tau
+C\int_{0}^{t}\|f\|_{B^{s-3}_{p,r}}\mbox{d}\tau.
\end{eqnarray}
Again, thanks to the product law in the Besov spaces
stated in Lemma \ref{21lem3}
and the embedding relation in Lemma \ref{21lem1},
one deduces
\begin{eqnarray}\label{unique-2}
\|\partial_{x}\mbox{sin}[u^{2}-u_{x}^{2}]\|_{B^{s-2}_{p,r}}
\leq \|\mbox{sin}(u^{2}-u_{x}^{2})\|_{B^{s-1}_{p,r}}
\leq \|u^{2}-u_{x}^{2}\|_{B^{s-1}_{p,r}}
\leq C\|u\|_{B^{s}_{p,r}}^{2}.
\end{eqnarray}
The Moser-type estimates in Lemma \ref{21lem4}
leads to
\begin{eqnarray}\label{unique-3}
&&\|[\mbox{cos}(u^{2}-u_{x}^{2})u_{x}m^{2}
-\mbox{cos}(v^{2}-v_{x}^{2})v_{x}n^{2}]
\|_{B^{s-3}_{p,r}}\nonumber\\
&&\leq C\|\{\mbox{cos}(u^{2}-u_{x}^{2})
-\mbox{cos}(v^{2}-v_{x}^{2})\}u_{x}m^{2}\|_{B^{s-3}_{p,r}}
+C\|\mbox{cos}(v^{2}-v_{x}^{2})(u_{x}-v_{x})m^{2}
\|_{B^{s-3}_{p,r}}\nonumber\\
&&\quad\quad+C\|\mbox{cos}(v^{2}-v_{x}^{2})v_{x}(m^{2}-n^{2})
\|_{B^{s-3}_{p,r}}\nonumber\\
&&\leq C\bigg\|\mbox{sin}\frac{u^{2}-u_{x}^{2}+v^{2}-v_{x}^{2}}
{2}\mbox{sin}\frac{u^{2}-u_{x}^{2}-v^{2}+v_{x}^{2}}
{2}u_{x}m^{2}\bigg\|_{B^{s-3}_{p,r}}\nonumber\\
&&\quad\quad+\|m\|_{B^{s-3}_{p,r}}
\|m\|_{B^{s-2}_{p,r}}
\|u_{x}-v_{x}\|_{B^{s-2}_{p,r}}
+\|v_{x}\|_{B^{s-2}_{p,r}}
\|m-n\|_{B^{s-3}_{p,r}}
\|m+n\|_{B^{s-2}_{p,r}}\nonumber\\
&&\leq C\|m\|_{B^{s-3}_{p,r}}
\|m\|_{B^{s-2}_{p,r}}
\|u_{x}\|_{B^{s-2}_{p,r}}
(\|u\|_{B^{s}_{p,r}}+\|v\|_{B^{s}_{p,r}})
\|u-v\|_{B^{s-1}_{p,r}}\nonumber\\
&&\quad\quad+C\|m\|_{B^{s-3}_{p,r}}
\|m\|_{B^{s-2}_{p,r}}
\|u-v\|_{B^{s-1}_{p,r}}
+C\|v\|_{B^{s}_{p,r}}
(\|u\|_{B^{s}_{p,r}}+\|v\|_{B^{s}_{p,r}})
\|u-v\|_{B^{s-1}_{p,r}}
\end{eqnarray}
and
\begin{eqnarray}\label{unique-4}
&&\|\{\mbox{sin}[u^{2}-u_{x}^{2}]
-\mbox{sin}(v^{2}-v_{x}^{2})\}n_{x}
\|_{B^{s-3}_{p,r}}\nonumber\\
&&\leq C\|n_{x}\|_{B^{s-3}_{p,r}}
\|\mbox{cos}[(u^{2}-u_{x}^{2}+v^{2}-v_{x}^{2})/2]
\mbox{sin}[(u^{2}-u_{x}^{2}-v^{2}+v_{x}^{2})/2]
\|_{B^{s-2}_{p,r}}\nonumber\\
&&\leq C\|n_{x}\|_{B^{s-3}_{p,r}}
\|u-v\|_{B^{s-1}_{p,r}}
(\|u\|_{B^{s}_{p,r}}^{2}
+\|v\|_{B^{s}_{p,r}}^{2}).
\end{eqnarray}
Substituting (\ref{unique-2})-(\ref{unique-4}) into (\ref{unique-1}) yields
\begin{eqnarray}\label{unique-5}
\|u-v\|_{B^{s-1}_{p,r}}\leq \|u_{0}-v_{0}\|_{B^{s-1}_{p,r}}
+C\int_{0}^{t}\|u-v\|_{B^{s-1}_{p,r}}
(\|u\|_{B^{s}_{p,r}}
+\|u\|_{B^{s}_{p,r}}^{4}
+\|v\|_{B^{s}_{p,r}}
+\|v\|_{B^{s}_{p,r}}^{4})\mbox{d}\tau.
\end{eqnarray}
The Gronwall inequality then implies $u=v$ or $m=n$.
We thus complete the proof of Theorem \ref{thmLocal}.
\end{proof}

\section{Blow-up  criterion}
 The proof of the  blow-up criterion for the solution to the Cauchy problem (\ref{sin-mCH}) will be provided in this section.

\begin{proof}
The proof of Theorem \ref{thmBlowup} will
be divided into  three steps.
The method is mainly induction with respect to $s$.
Let us recall that equation (\ref{sin-mCH})
can be recast as
\begin{eqnarray}\label{blowup-1}
m_{t}+\mbox{sin}(u^{2}-u_{x}^{2})\partial_{x}m
=-2\mbox{cos}(u^{2}-u_{x}^{2})u_{x}m^{2}.
\end{eqnarray}
\textbf{Step 1}:
When $s\in (1/2,1)$,
applying  Lemma \ref{22lem5} to (\ref{blowup-1})
yields
\begin{eqnarray}\label{blowup-2}
\|m\|_{H^{s}}
\leq \|m_{0}\|_{H^{s}}
+C\int_{0}^{t}\|m\|_{H^{s}}\|
\partial_{x}\mbox{sin}(u^{2}-u_{x}^{2})
\|_{L^{\infty}}\mbox{d}\tau
+C\int_{0}^{t}
\|2\mbox{cos}(u^{2}-u_{x}^{2})u_{x}m^{2}\|_{H^{s}}
\mbox{d}\tau.
\end{eqnarray}
Using
 $u=(1-\partial_{x}^{2})^{-1} m=p * m$
and $\|p\|_{L^{1}}=\|\partial_{x} p\|_{L^{1}}=1$,
one finds after employing the Young inequality
that for   $s \in \mathbb{R}$
\begin{eqnarray}\label{blowup1}
\begin{array}{l}
\|u\|_{L^{\infty}}+\left\|u_{x}\right\|_{L^{\infty}}+\left\|u_{x x}\right\|_{L^{\infty}} \leq C\|m\|_{L^{\infty}},\v\\
\|u\|_{H^{s}}+\left\|u_{x}\right\|_{H^{s}}+\left\|u_{x x}\right\|_{H^{s}} \leq C\|m\|_{H^{s}}.
\end{array}
\end{eqnarray}
Using  (\ref{blowup1}), one derives
\begin{eqnarray}\label{blowup-3}
\|
\partial_{x}\mbox{sin}[u^{2}-u_{x}^{2}]
\|_{L^{\infty}}
=\|2\mbox{cos}[u^{2}-u_{x}^{2}]u_{x}m
\|_{L^{\infty}}
\leq C\|m\|_{L^{\infty}}^{2}
\end{eqnarray}
and
\bee\label{blowup-4}
\begin{array}{rl}
\|2\mbox{cos}[u^{2}-u_{x}^{2}]u_{x}m^{2}\|_{H^{s}}
&\leq C\|u_{x}m^{2}\|_{H^{s}}\nonumber\\
&\leq C\|m\|_{L^{\infty}}\|u_{x}m\|_{H^{s}}
+C\|m\|_{H^{s}}\|u_{x}m\|_{L^{\infty}}\nonumber\\
&\leq C\|m\|_{L^{\infty}}^{2}\|m\|_{H^{s}}.
\end{array}
\ene
Combining (\ref{blowup-2}), (\ref{blowup-3})
and (\ref{blowup-4}), one finds
\begin{eqnarray}\label{blowup-5}
\|m\|_{H^{s}}
\leq \|m_{0}\|_{H^{s}}
+C\int_{0}^{t}\|m\|_{H^{s}}
\|m\|_{L^{\infty}}^{2}
\mbox{d}\tau
\end{eqnarray}
or
\begin{eqnarray}\label{blowup-6}
\|m\|_{H^{s}}
\leq \|m_{0}\|_{H^{s}}
e^{C\int_{0}^{t}
\|m\|_{L^{\infty}}^{2}
\mbox{d}\tau}.
\end{eqnarray}
Consequently, if $\int_{0}^{T^{*}}\|m(\tau)\|_{L^{\infty}}^{2} \mbox{d} \tau<\infty$ for the maximal existence time $T^{*}<\infty$,
then the inequality (\ref{blowup-6}) implies that
$\underset{t \rightarrow T^{*}}{\limsup }\|m(t)\|_{H^{s}}<\infty,$
which contradicts our assumption on $T^{*}$. Thus we complete the proof of Theorem \ref{thmBlowup}
for the case $s\in (1/2,1).$

\textbf{Step 2}:
For $s\in [1,2),$
after differentiating (\ref{blowup-1})
once with respect to $x$, we obtain
\begin{eqnarray}\label{blowup-7}
m_{xt}+\mbox{sin}(u^{2}-u_{x}^{2})\partial_{x}m_{x}=-2\mbox{cos}(u^{2}-u_{x}^{2})u_{x}mm_{x}
-2\partial_{x}[\mbox{cos}(u^{2}-u_{x}^{2})u_{x}m^{2}].
\end{eqnarray}
Lemma \ref{22lem5} again yields
\begin{eqnarray}\label{blowup-8}
&&\|m_{x}\|_{H^{s-1}}
\leq \|m_{0x}\|_{H^{s-1}}
+C\int_{0}^{t}\|m_{x}\|_{H^{s-1}}
\|
\mbox{sin}(u^{2}-u_{x}^{2})
\|_{L^{\infty}}\mbox{d}\tau\\
&&\quad\quad\quad\quad\quad\quad
+C\int_{0}^{t}
\|-2\mbox{cos}(u^{2}-u_{x}^{2})u_{x}mm_{x}
-2\partial_{x}[\mbox{cos}(u^{2}-u_{x}^{2})u_{x}m^{2}]
\|_{H^{s-1}}
\mbox{d}\tau.\nonumber
\end{eqnarray}
Simple computation leads to
\begin{eqnarray}\label{blowup-9}
&&\|-2\mbox{cos}(u^{2}-u_{x}^{2})u_{x}mm_{x}
-2\partial_{x}[\mbox{cos}(u^{2}-u_{x}^{2})u_{x}m^{2}]
\|_{H^{s-1}}\nonumber\\
&&\quad\quad\leq C\|u_{x}mm_{x}\|_{H^{s-1}}
\|u_{x}m^{2}\|_{H^{s}}\nonumber\\
&&\quad\quad\leq C\|\partial_{x}(u_{x}m^{2})\|_{H^{s-1}}
+C\|u_{xx}m^{2}\|_{H^{s-1}}
+C\|m\|_{H^{s}}\|m\|_{L^{\infty}}^{2}\nonumber\\
&&\quad\quad\leq C\|m\|_{H^{s}}\|m\|_{L^{\infty}}^{2}.
\end{eqnarray}
Combining (\ref{blowup-3}) and (\ref{blowup-8})-(\ref{blowup-9}),
one deduces
\begin{eqnarray}\label{blowup-10}
\|m_{x}\|_{H^{s-1}}
\leq \|m_{0x}\|_{H^{s-1}}
+C\int_{0}^{t}\|m\|_{H^{s}}
+\|m\|_{L^{\infty}}^{2}\mbox{d}\tau.
\end{eqnarray}
Adapting the same argument as in \textbf{Step 1}, we
conclude that this theorem holds for $s\in [1,2).$

\textbf{Step 3}:
Assume $2 \leq l \in \mathbb{N}.$
Suppose (\ref{blowupTime}) holds when $l-1 \leq$ $s<l.$
Using induction, we should prove the validity of (\ref{blowupTime})
 for $l\leq s<l+1.$
 Applying $\partial_{x}^{l}$ to (\ref{blowup-1}) leads to
\begin{eqnarray}\label{blowup-11}
&&\partial_{t}\partial_{x}^{l}m
+\mbox{sin}(u^{2}-u_{x}^{2})\partial_{x}^{l+1}m=-\sum_{i=0}^{l-1}C_{l}^{i}
\partial_{x}^{l-i}\mbox{sin}(u^{2}-u_{x}^{2})
\partial_{x}^{i+1}m
-2\partial_{x}^{l}[\mbox{cos}(u^{2}-u_{x}^{2})u_{x}m^{2}]
\coloneqq f_{2}.\nonumber
\end{eqnarray}
Lemma \ref{22lem5} once again leads to
\begin{eqnarray}\label{blowup-12}
&&\|\partial_{x}^{l}m\|_{H^{s-l}}
\leq \|m_{0}\|_{H^{s}}
+C\int_{0}^{t}\|\partial_{x}^{l}m\|_{H^{s-l}}
\|\partial_{x}\mbox{sin}(u^{2}-u_{x}^{2})\|_{L^{\infty}}\mbox{d}\tau
+C\int_{0}^{t}\|f_{2}\|_{H^{s-l}}\mbox{d}\tau.
\end{eqnarray}

Using Lemma \ref{21lem5} and the Sobolev embedding inequality
produces
\begin{eqnarray}\label{blowup-13}
&&\|\sum_{i=0}^{l-1}C_{l}^{i}
\partial_{x}^{l-i}\mbox{sin}(u^{2}-u_{x}^{2})
\partial_{x}^{i+1}m\|_{H^{s-l}}\nonumber\\
&&\leq \sum_{i=0}^{l-1}C_{l}^{i}
\|\partial_{x}^{l-i}\mbox{sin}(u^{2}-u_{x}^{2})
\partial_{x}^{i+1}m\|_{H^{s-l}}\nonumber\\
&&\leq \bigg(
\|\partial_{x}^{l-i}\mbox{sin}(u^{2}-u_{x}^{2})\|_{H^{s-l+1}}
\|\partial_{x}^{i}m\|_{L^{\infty}}+\|\partial_{x}^{l-i}\mbox{sin}(u^{2}-u_{x}^{2})\|_{L^{\infty}}
\|\partial_{x}^{i+1}m\|_{H^{s-l}}
\bigg)\nonumber\\
&&\leq C\|m\|_{H^{l+1/2+\epsilon}}
\|\mbox{sin}(u^{2}-u_{x}^{2})\|_{H^{s-i+1}}+C\|m\|_{H^{s-l+i+1}}
\|\mbox{sin}(u^{2}-u_{x}^{2})\|_{H^{l-i+1/2+\epsilon}}\nonumber\\
&&\leq C\|m\|_{H^{l+1/2+\epsilon}}
\|u^{2}-u_{x}^{2}\|_{H^{s-i+1}}+C\|m\|_{H^{s-l+i+1}}
\|u^{2}-u_{x}^{2}\|_{H^{l-i+1/2+\epsilon}}\nonumber\\
&&\leq C\|m\|_{H^{s-1/2+\epsilon}}^{2}
\|m\|_{H^{s}},
\end{eqnarray}
where $\epsilon\in (0,1/8)$ so that $H^{\frac{1}{2}+\epsilon}(\mathbb{R}) \hookrightarrow L^{\infty}(\mathbb{R})$.

Direct computation gives
\begin{eqnarray}\label{blowup-14}
\|2\partial_{x}^{l}[\mbox{cos}(u^{2}-u_{x}^{2})u_{x}m^{2}]
\|_{H^{s-l}}
\leq C\|u_{x}m^{2}\|_{H^{s}}
\leq C\|m\|_{H^{s}}
\|m\|_{H^{l-1/2+\epsilon}}^{2}.
\end{eqnarray}
(\ref{blowup-12}) together with
(\ref{blowup-13})-(\ref{blowup-14}) yields
\begin{eqnarray}\label{blowup-15}
&&\|\partial_{x}^{l}m\|_{H^{s-l}}
\leq \|m_{0}\|_{H^{s}}
+C\int_{0}^{t}\|m\|_{H^{s}}
\|m\|_{H^{l-1/2+\epsilon}}^{2}\mbox{d}\tau
\end{eqnarray}
or
\begin{equation}\label{blowupk6}
\|m(t)\|_{H^{s}}
\leq\left\|m_{0}\right\|_{H^{s}}
 \exp \left\{C \int_{0}^{t}
 \left(\|m(\tau)\|_{H^{l-\frac{1}{2}+\epsilon}}^{2}+1\right) \mbox{d} \tau\right\}.
\end{equation}
Therefore, if the maximal existence time $T^{\star}<\infty$ satisfies
$\int_{0}^{T^{\star}}\|m(\tau)\|_{L^{\infty}}^{2} \mbox{d} \tau<\infty,$
then the uniqueness of the solution provided by
Theorem \ref{thmLocal} ensures the
uniform boundedness of
$\|m(t)\|_{H^{l-\frac{1}{2}+\epsilon}}$ in
$t \in (0, T^{*})$ recalling our   induction assumption, which together  with (\ref{blowupk6}) implies
$\limsup _{t \rightarrow T^{*}}\|m(t)\|_{H^{s}}<\infty,$ a contradiction. We thus complete the proof of Theorem \ref{thmBlowup}.
\end{proof}

\section{Wave-breaking}
In this section,
we will
prove Theorem \ref{waveBreaking}.
Before doing this, we shall
  first deduce  the precise blow-up quantity for strong solutions of (\ref{sin-mCH}).

Let $q(t, x)$ solve the following ordinary differential equation:
\begin{eqnarray}\label{character}
\begin{cases}
 \dfrac{\mbox{d}}{\mbox{d}t}q(t,x)
=\mbox{sin}(u^{2}-u_{x}^{2})(t,q(t,x)),\\
q(x,0)=x.
\end{cases}
\end{eqnarray}
After differentiating (\ref{character})
with respect to $x$,
 we obtain

\begin{lem}\label{lemDiffermophism}
Suppose $u_{0} \in H^{s}(\mathbb{R})$  with $s>\frac{5}{2}$.  Let $T^{*}>0$ be the maximal existence time of the
 solution $u$ to Eq.~(\ref{sin-mCH}).
   Then there exists a  unique solution $q \in C^{1}([0, T^{*}) \times \mathbb{R}; \mathbb{R})$  to Eq.~(\ref{character}) satisfying
\begin{eqnarray}\label{conservative2}
q_{x}(t,x)
 =\exp\bigg\{2\int_{0}^{t}\cos(u^{2}-u_{x}^{2})mu_{x}\mbox{d}s
 \bigg\}>0.
\end{eqnarray}
Moreover, we have the expression of $m(t, q(t, x))$ as
\begin{eqnarray}\label{conservative21}
m(t,q(t,x))
=m_{0}(x)\exp\bigg\{-2\int_{0}^{t}\cos(u^{2}-u_{x}^{2})mu_{x}(s,q(s,x))\mbox{d}s
 \bigg\},
\end{eqnarray}
which implies that the
sign and zeros of $m(x,t)$ are
the same as those of $m_{0}(x).$
\end{lem}

\begin{proof}
Applying $\partial_{x}$ to Eq.~(\ref{character}) leads to
\begin{eqnarray}\label{conservative3}
\begin{cases}
\dfrac{\mbox{d}}{\mbox{d}t}q_{x}(t,x)
=2\mbox{cos}(u^{2}-u_{x}^{2})u_{x}m(t,q(t,x))q_{x}(t,x),\v\\
q_{x}(0,x)=1.
\end{cases}
\end{eqnarray}
Solving (\ref{conservative3}) produces the solution given by (\ref{conservative2}).
The Sobolev embedding inequality
gives for $T^{\prime}<T^{*}$
\begin{eqnarray*}
\sup_{(s,x)\in [0,T^{\prime})\times \mathbb{R}}
|2\mbox{cos}(u^{2}-u_{x}^{2})u_{x}m(s,x)|<\infty,
\end{eqnarray*}
which along with (\ref{conservative2}) yields $q_{x}(t,x)\geq \exp(-Ct),\, (t,x)\in [0,T^*)\times \mathbb{R}$ for some $C>0$,
and consequently
$q(t, \cdot)$ is an increasing diffeomorphism of $\mathbb{R}$ before the blow-up time,
that is, (\ref{conservative2}) holds.

It follows from Eqs.~(\ref{sin-mCH}) and (\ref{character}) that
\begin{eqnarray}\label{conservative6}
\frac{\mbox{d}}{\mbox{d}t}m(t, q(t,x))
&=&m_{t}(t, q(t,x))
+m_{x}(t, q(t,x))q_{t}(t,x)\nonumber\\
&=&m_{t}(t, q(t,x))
+m_{x}(t, q(t,x))
[\mbox{sin}(u^{2}-u_{x}^{2})](t, q(t,x))\nonumber\\
&=&[m_{t}
+m_{x}(\mbox{sin}(u^{2}-u_{x}^{2}))](t, q(t,x))\nonumber\\
&=&[-2\mbox{cos}(u^{2}-u_{x}^{2})u_{x}m](t,q(t,x))m(t, q(t,x)).
\end{eqnarray}
Solving Eq.~(\ref{conservative6}) yields Eq.~(\ref{conservative21}). We thus complete the proof of Lemma \ref{lemDiffermophism}.
\end{proof}
We next deduce the precise blow-up quantity
of the Cauchy problem (\ref{sin-mCH}).

\begin{lem}\label{blowupQuantityThm} Suppose $u_{0} \in H^{s}(\mathbb{R})$ with $s>\frac{5}{2}$. Let  $T^{*}>0$ be the maximal existence time of the solution  $u$ to the Cauchy problem (\ref{sin-mCH}).  Then the solution $u$ blows up in finite time if and only if
\begin{eqnarray}\label{blowupQuantity}
  \liminf _{t \rightarrow T^{*}}\left(\inf _{x \in \mathbb{R}}\left(2\cos(u^{2}-u_{x}^{2})u_{x}m(t,x)\right)\right)=-\infty.
\end{eqnarray}
\end{lem}
\begin{proof}
From the expression of $m(t,q(t,x))$   in Eq.~(\ref{conservative21}),
  we conclude that
  if there exists a positive constant $K_{1}$
  such that
  \begin{equation*}
  \inf _{x \in \mathbb{R}} (-2\cos(u^{2}-u_{x}^{2})u_{x}m(t,x))
  \geq -K_{1}, \quad 0\leq t\leq T^{*},
  \end{equation*}
then
  \begin{equation*}
  \begin{aligned}
\|m(t)\|_{L^{\infty}}
=\|m(t, q(t, x))\|_{L^{\infty}}= m_{0}(x)
\exp \left(-2 \int_{0}^{t}\left(\mbox{cos}(u^{2}-u_{x}^{2})
u_{x}m(\tau, q(\tau, x)) \mbox{d}\tau\right)\right)  \leq e^{ K_{1} T^{*}}\|m_{0}\|_{L^{\infty}},
\end{aligned}
  \end{equation*}
  which combined with Theorem \ref{thmBlowup}
  implies that $m(t,x)$ will not blow up in a finite time.

  However, if (\ref{blowupQuantity}) holds true,
  then the Sobolev embedding
  ensure that $m(t,x)$ will  blow up in a finite time.
  We thus complete the proof
   of Lemma  \ref{blowupQuantityThm}.
\end{proof}
Now we are in a position  to
prove Theorem \ref{waveBreaking}.
\begin{proof}[Proof of Theorem \ref{waveBreaking}]
Set the precise blow-up quantity as
\begin{eqnarray}\label{M}
M=\mbox{cos}(u^{2}-u_{x}^{2})mu_{x}.
\end{eqnarray}
Next, we  need to  estimate  the dynamics of $M$
along the characteristic.
To do this, we first compute
\begin{eqnarray}\label{u}
&&(1-\partial_{x}^{2})
[u_{t}+
\mbox{sin}(u^{2}-u_{x}^{2})u_{x}]\nonumber\\
&&=m_{t}+(1-\partial_{x}^{2})
[\mbox{sin}(u^{2}-u_{x}^{2})u_{x}]\nonumber\\
&&=-\mbox{sin}(u^{2}-u_{x}^{2})
(u_{x}-u_{xxx})
-2\mbox{cos}(u^{2}-u_{x}^{2})u_{x}m^{2}+\mbox{sin}(u^{2}-u_{x}^{2})u_{x}
-\partial_{x}^{2}
[\mbox{sin}(u^{2}-u_{x}^{2})u_{x}]\nonumber\\
&&=\mbox{sin}(u^{2}-u_{x}^{2})u_{xxx}
-2\mbox{cos}(u^{2}-u_{x}^{2})u_{x}m^{2}
-\partial_{x}^{2}
[\mbox{sin}(u^{2}-u_{x}^{2})u_{x}]\nonumber\\
&&=-2uM-2\partial_{x}(u_{x}M),
\end{eqnarray}
from which we derive
\begin{eqnarray}\label{u-1}
u_{t}+
\mbox{sin}(u^{2}-u_{x}^{2})u_{x}
=-2p*(uM)
-2p_{x}*(u_{x}M).
\end{eqnarray}

Differentiating (\ref{u-1})
with respect to $x$ then produces
\begin{eqnarray}\label{ux}
u_{xt}+
\mbox{sin}(u^{2}-u_{x}^{2})u_{xx}
=-2p_{x}*(uM)
-2p*(u_{x}M).
\end{eqnarray}

Recasting (\ref{sin-mCH}) leads to
\begin{eqnarray}\label{m}
m_{t}+\mbox{sin}(u^{2}-u_{x}^{2})m_{x}=-2mM.
\end{eqnarray}
Combining (\ref{u-1})-(\ref{m}), one obtains
\begin{eqnarray}\label{M-1}
&&M_{t}+\mbox{sin}(u^{2}-u_{x}^{2})M_{x}\nonumber\\
&&=-\mbox{sin}(u^{2}-u_{x}^{2})mu_{x}
\{2u[-2p*(uM)
-2p_{x}*(u_{x}M)]
-2u_{x}[-2p_{x}*(uM)
-2p*(u_{x}M)]\}\nonumber\\
&&\quad\quad+\mbox{cos}(u^{2}-u_{x}^{2})u_{x}(-2mM)
+\mbox{cos}(u^{2}-u_{x}^{2})m
[-2p_{x}*(uM)
-2p*(u_{x}M)]\nonumber\\
&&= 4\mbox{sin}(u^{2}-u_{x}^{2})muu_{x}
[p*(uM)+p_{x}*(u_{x}M)]-4\mbox{sin}(u^{2}-u_{x}^{2})mu_{x}^{2}
[p_{x}*(uM)+p*(u_{x}M)]\qquad \nonumber\\
&&\quad\quad-2M^{2}
-2\mbox{cos}(u^{2}-u_{x}^{2})m
[p_{x}*(uM)+p*(u_{x}M)].
\end{eqnarray}
The positivity of the initial data $m_{0}$
and the equality $u=p*m$
yields $|u_{x}|\leq u$, accordingly, one obtains
\begin{eqnarray}
(\ref{M-1})\leq -2M^{2}
+C(\|u\|_{L^{\infty}}^{5}
+\|u\|_{L^{\infty}}^{3})m,
\end{eqnarray}
which combined with the definition of $\bar{M}(t)$
gives
\begin{eqnarray}\label{M-2}
\frac{\mbox{d}}{\mbox{d}t}\bar{M}(t)
\!\!\!&&=(M_{t}+\mbox{sin}(u^{2}-u_{x}^{2})M_{x})(t, q(t,x_{0}))\nonumber\\
&&\leq -2\bar{M}^{2}(t)
+C(\|u_{0}\|_{H^{1}}^{5}
+\|u_{0}\|_{H^{1}}^{3})\bar{m}\nonumber\\
&&\leq -2\bar{M}^{2}(t)
+C_{1}\bar{m},
\end{eqnarray}
where $C_{1}=C(\|u_{0}\|_{H^{1}}^{5}
+\|u_{0}\|_{H^{1}}^{3}).$

On the other hand, one easily deduces
\begin{eqnarray}\label{M-3}
\frac{\mbox{d}}{\mbox{d}t}\bar{m}(t)
=-2\bar{m}\bar{M}.
\end{eqnarray}
From (\ref{M-2})-(\ref{M-3}), we have
\bee\label{M-4}
\begin{array}{rl}
\dfrac{\mbox{d}}{\mbox{d}t}
\bigg(\dfrac{\bar{M}(t)}{\bar{m}(t)}\bigg)
&=\dfrac{1}{\bar{m}^{2}(t)}\{\bar{M}^{\prime}(t)\bar{m}(t)
-\bar{M}(t)\bar{m}^{\prime}(t)\}\v\\
&\leq \dfrac{1}{\bar{m}^{2}(t)}
\{\bar{m}(t)(-2\bar{M}^{2}(t)+C_{1}\bar{m})
-\bar{M}(-2\bar{m}\bar{M})
\}
=C_{1}.
\end{array}
\ene
Integrating (\ref{M-4}) from 0 to $t$ yields
\begin{eqnarray}\label{M-5}
\frac{\bar{M}(t)}{\bar{m}(t)}
\leq \frac{\bar{M}(0)}{\bar{m}(0)}
+C_{1}t
\end{eqnarray}
or
\begin{eqnarray}\label{M-6}
\bar{M}(t)
\leq \bigg(\frac{\bar{M}(0)}{\bar{m}(0)}
+C_{1}t\bigg)\bar{m}(t).
\end{eqnarray}
Consequently, one derives
\begin{eqnarray}\label{M-7}
&&\frac{\mbox{d}}{\mbox{d}t}
\bigg(\frac{1}{\bar{m}(t)}\bigg)
=-\frac{1}{\bar{m}^{2}(t)}
(-2\bar{m}\bar{M})
=2\frac{\bar{M}(t)}{\bar{m}(t)}
\leq 2\bigg(\frac{\bar{M}(0)}{\bar{m}(0)}
+C_{1}t\bigg).
\end{eqnarray}

Integrating once again, we obtain
\begin{eqnarray}\label{M-8}
0<
\frac{1}{\bar{m}(t)}
\leq 2\bigg(\frac{\bar{M}(0)}{\bar{m}(0)}t
+\frac{1}{2}C_{1}t^{2}\bigg)
+\frac{1}{\bar{m}(0)}
\coloneqq h(t).
\end{eqnarray}
Since $\bar{M}(0)<0$,
then $h^{\prime}(0)<0$.
  According to the fact that
   $\lim _{t \rightarrow \infty} h^{\prime}(t)=+\infty$ and the continuity of the function $h^{\prime}(t)$,
   there exists a $\xi>0$ such that $h^{\prime}(\xi)=0.$
    Under the assumption (\ref{assumption}), we have $h(\xi)<0$.
    Note that $h(0)=\frac{1}{\bar{m}(0)}>0$
    and $h(t) \in C[0,+\infty)$, one can find some $T^{*} \in(0, \xi)$ such that
\begin{equation}\label{1overm}
0<\frac{1}{\bar{m}(t)}
\leq
 h(t) \rightarrow 0, \quad \text { as } t \rightarrow T^{*},
\end{equation}
which indicates that $\lim _{t \rightarrow T^{*}} \bar{m}(t)=+\infty$.
Since
 $\frac{\bar{M}(0)}{\bar{m}(0)}
+C_{1}T^{*}<\frac{\bar{M}(0)}{\bar{m}(0)}
+C_{1}\xi=0$, then by
(\ref{M-6}) and (\ref{1overm}),
we derive
$$
\inf _{x \in \mathbb{R}} [\mbox{cos}(u^{2}-u_{x}^{2})mu_{x}](t, x) \rightarrow-\infty, \quad \text { as } t \rightarrow T^{*}.
$$
Using   Lemma \ref{blowupQuantityThm},
we conclude that
the solution $m$ blows up at the time $T^{*} \in(0, \xi]$.
We thus  complete the proof of Theorem \ref{waveBreaking}.
\end{proof}

\vspace{0.15in}
\noindent \textbf{Declarations}

\vspace{0.15in}
\noindent \textbf{Funding}\, This work was partially supported by the NSF of China under Grants Nos. 11925108 and 11731014.

\vspace{0.15in}
\noindent \textbf{Conflicts of interest/Competing interests}\, There are no conflicts of interest.

\vspace{0.15in}
\noindent \textbf{Availability of data and material}\, The manuscript has no  associated data.

\vspace{0.15in}
\noindent \textbf{Code availability}\, The manuscript has no  associated code.

\vspace{0.15in}
\noindent \textbf{Authors' contributions}\, These authors contributed equally to this work.

\end{document}